\documentclass[11pt,a4paper,svgnames,x11names,reqno]{amsart}
\usepackage{amssymb}
\usepackage[utf8]{inputenc}
\usepackage[T1]{fontenc}

\usepackage[inline]{enumitem}
\setlist[enumerate,itemize]{leftmargin=*}
\usepackage{qtree}
\usepackage{tikz-cd} 

\theoremstyle{plain}
\newtheorem{theorem}{Theorem}[section]
\newtheorem{lemma}[theorem]{Lemma}

\newtheorem{proposition}[theorem]{Proposition}
\newtheorem{remark}[theorem]{Remark}
\theoremstyle{definition}
\newtheorem{definition}{Definition}[section]
\usepackage{mathtools}
\RequirePackage{etoolbox}
\DeclarePairedDelimiterX\abs[1]{\lvert}{\rvert}{\ifblank{#1}{\:\cdot\:}{#1}}
\DeclarePairedDelimiterX\braces[1]{\lbrace}{\rbrace}{\ifblank{#1}{\:\cdot\:}{#1}}

\usepackage{url}
\usepackage[colorlinks,linkcolor=blue,citecolor=blue]{hyperref}

\title[Tree tilings and Cantor homeomorphisms]%
{Tilings of the infinite $p$-ary tree and Cantor homeomorphisms}

\author[A. Cobos]{Alberto Cobos}
\address{%
		School of Mathematics and Statistics,
		Hicks Building,
		Hounsfield Road,
		Sheffield S3 7RH,
		United Kingdom.
	   }
\email{acobosrabano1@sheffield.ac.uk}

\author[L.~M.~Navas]{Luis M. Navas}
\address{%
       Departamento de Matem\'aticas,
       Universidad de Salamanca,
       37008 Salamanca, Spain.}
\email{navas@usal.es}

\thanks{The research of the second author is supported by Grants PGC2018-099599-B-I00 and PGC2018-096504-B-C32 of the MICINN (Spain) and Grant J416/463AC03 of the JCyL}


\begin{document}


\begin{abstract}
We define a notion of tiling of the full infinite $p$-ary tree, establishing a series of equivalent criteria for a subtree to be a tile, each of a different nature; namely, geometric, algebraic, graph-theoretic, order-theoretic, and topological. We show how  these results can be applied in a straightforward and constructive manner to define homeomorphisms between two given spaces of $p$-adic integers, $\mathbb{Z}_{p}$ and $\mathbb{Z}_{q}$, endowed with their corresponding standard non-archimedean metric topologies.
\end{abstract}

\maketitle


\section{Introduction.}
\label{sec:introduction}

This paper arose out of curiosity regarding how easily one might write down an explicit homeomorphism between two given spaces of $p$-adic integers, $\mathbb{Z}_{p}$ and $\mathbb{Z}_{q}$, for two primes $p,q$, including auto-homeomorphisms when $p=q$. The existence of a homeomorphism follows from the fact that they are Cantor spaces (compact, metrizable, totally disconnected, perfect topological spaces). However, such maps cannot preserve the algebraic structure.\footnote{For example, there are no non-trivial additive group homomorphisms between $\mathbb{Z}_{p}$ and $\mathbb{Z}_{q}$ for distinct primes $p,q$, and no non-trivial ring endomorphisms of $\mathbb{Z}_{p}$.}

Robert in~\cite[Ch. 2]{Robert} describes various ways of mapping $\mathbb{Z}_{p}$ homeomorphically to a fractal subset of a Euclidean space $\mathbb{R}^{n}$, beginning with the well-known correspondence between $\mathbb{Z}_{2}$ and the classical Cantor set $C$, with a view towards geometric visualization. 

Of course, having fixed a homeomorphism $\mathbb{Z}_{p} \to \mathbb{Z}_{q}$, any other is given by composition with an auto-homeomorphism of $\mathbb{Z}_{p}$, so in principle we could just study these, and similarly, fixing homeomorphisms $\mathbb{Z}_{p} \to C$ and $\mathbb{Z}_{q} \to C$, we would be looking at the auto-homeomorphism group of $C$, which has been extensively studied in the literature, especially that relating to dynamical systems (e.g.~\cite{AKW,BernardesDarji}).

However, our interest was more down to earth, in the sense that we were looking to explicitly ``directly connect'' two different primes $p,q$, somewhat in the spirit of~\cite[\S 2.6]{Robert} which gives a map $\mathbb{Z}_{2} \times \mathbb{Z}_{2} \to \mathbb{Z}_{3}$, and to do so in as simple a manner as possible, whatever that might actually turn out to mean. A different example of an explicit auto-homeomorphism of $C$ can be found in~\cite{Kimura}.

The eventual solution was formulated, somewhat to our surprise, in terms of tilings of the full infinite $p$-ary tree $T_{p}$. That $T_{p}$ should be involved is not the surprising part, since the construction of $\mathbb{Z}_{p}$ as the inverse limit $\varprojlim \mathbb{Z}/p^{n}\mathbb{Z}$ under the mod $p^{n}$ reduction maps leads to Hensel's representation of a $p$-adic integer by an infinite series $\sum_{n=0}^{\infty} a_{n} p^{n}$ where $a_{n} \in \braces{0,1,\ldots,p-1}$ and geometrically this may be visualized as a path in $T_{p}$, in much the same way that the base-$p$ expansion of a real number determines a sequence of nested intervals collapsing to a point.

In the course of discovering and proving the main results, we found it useful to relate the geometry of such tilings to operations in the monoid of words over the alphabet with $p$ symbols, and finally to the special characteristics of the $p$-adic topology.

To our knowledge, the results presented here are new. A few similar constructions, arising from the relation of Cantor spaces to inverse limit spaces via sequences of nested refinements of a partition, can be found in~\cite{BernardesDarji}, although the theme of that paper is the graph-theoretic structure of the auto-homeomorphism group of $C$.

In much of the discussion, it is not necessary that $p$ be a prime, but when considering the topology, it is convenient to have the standard $p$-adic metric on hand to relate words with $p$-adic balls and partitions formed by them.

In Section~\ref{sec:tilings}, we give our definition and classification of tilings for the full infinite $p$-ary tree, which we hope will be interesting in itself. Our main result is Theorem~\ref{T:tile equivalence}, which establishes the equivalence of various properties which are \emph{a priori} of a different nature; namely, the geometric property of being a tile (Definition~\ref{D:tiling}), the graph-theoretic property of fullness, the order property of leaf-comparability (Lemma~\ref{L:leaf-comparabilty}), an algebraic property of unique factorization and finally, the topological property of being a partition (Theorem~\ref{T:tiling topological characterization}).

In Section~\ref{sec:homeomorphisms}, we interpret the main result in terms of the $p$-adic topology and use such tilings to fulfill our promise of a relatively easy construction of explicit homeomorphisms between $\mathbb{Z}_{p}$ and $\mathbb{Z}_{q}$. We will conclude with some examples, including auto-homeomorphisms of $\mathbb{Z}_{2}$.

\begin{figure}[ht]
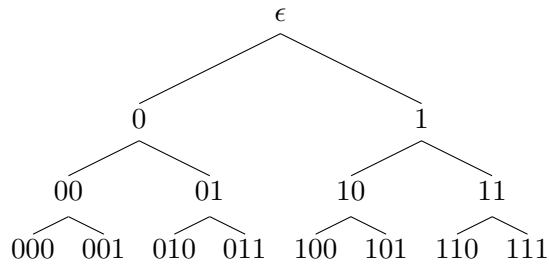

\begin{center}
\scalebox{1}{\Tree%
[
	.$\epsilon$ 	
	[
		.$0$ [.$00$ 
					[.$000$ ][.$001$ ]
			 ]
			 [
			  .$01$ 
					[.$010$ ][.$011$ ]
			 ]
	]
	[
		.$1$ [.$10$ 
					[.$100$ ][.$101$ ]
			 ]
			 [
			   .$11$
					[.$110$ ][.$111$ ]
			 ]
	]
]
}
\caption{The first few levels of the binary tree $T_{2}$, representing both  $\mathbb{Z}_{2}$ and the Cantor set $C$ (after mapping the ternary digit $2$ to $1$).}
\label{fig:binary tree}
\end{center}
\end{figure}

\section{Tilings of the full infinite $p$-ary tree}
\label{sec:tilings}

Given a graph $G$, let $V(G)$ and $E(G)$ denote its vertex and edge sets respectively. Since we will be concerned with rooted trees, we will also refer to nodes, children and parents, etc. In particular for a subtree $S$ we will denote by $L(S)$ its collection of \emph{leaves}, that is to say, the childless nodes of $S$. Recall that a $p$-ary tree is one in which each node has at most $p$ children. The case where every non-leaf node has the maximum possible number of children is particularly important.

\begin{definition}
\label{D:full}
A $p$-ary tree is \emph{full} if every node has either no children or the maximum number $p$ of possible children. In the latter case we shall say that the node \emph{splits completely.}
\end{definition}

We are primarily interested in the full infinite $p$-ary rooted tree $T_{p}$. For our purposes it will be convenient to identify the vertex set $V(T_{p})$ with the set of $p$-ary words $\Sigma_{p}^{*}$, consisting of finite strings over the alphabet $\Sigma_{p} = \braces{0,1,2,\dots,p-1}$, as we see in Figure~\ref{fig:binary tree} for $p=2$. Thus the set of edges $E(T_{p})$ consists of pairs $e = (w,wc)$ where $w \in \Sigma_{p}^{*}$ and $c \in \Sigma_{p}$ is a single character. The root of $T_{p}$ is the empty word $\epsilon$. We will then refer to depth or level in relation to $\epsilon$, considered as level $0$, and accordingly we may speak of the top or root node of a subtree.

It is well-known that $\Sigma_{p}^{*}$ is a cancellative monoid with respect to concatenation, denoted by juxtaposition, with $\epsilon$ as identity element. We will denote the length of a word $w$ by $\abs{w}$, with $\abs{\epsilon} = 0$. The length is a monoid morphism $\Sigma_{p}^{*} \to \mathbb{N}_{0}$, namely $\abs{vw} = \abs{v} + \abs{w}$.

Recall that $v$ is a prefix of $w$ if $w = v w'$ for some word $w'$, i.e. if $v$ is a left factor of $w$. The empty word is a prefix of every word. The prefix relation is a partial order on words which we will denote by $\leq$. 

In our model of $T_{p}$, the prefix relation coincides with the ancestor relation, namely $v \leq w$ if and only if $v$ is an ancestor of $w$, or equivalently $w$ is a descendant of $v$. In particular, if $w = c_{0}c_{1} \cdots c_{n-1} \in \Sigma_{p}^{*}$, the unique path from $w$ to the root node $\epsilon$ is obtained by removing one character at a time from $w$ proceeding from right to left: $c_{0} \cdots c_{n-1} \to c_{0} \cdots c_{n-2} \to \ldots \to c_{0} \to \epsilon$ (see Figure~\ref{fig:binary tree}).

Note that $\Sigma_{p}^{*}$ acts on $T_{p}$ on the left via prefixing, i.e. left multiplication on itself as vertex set, with the induced action on edges, $w (v,v') = (w v, w v')$. We will refer to this action as left translation or simply translation.

\begin{definition}
\label{D:almost disjointness}
We will say that two subgraphs of a graph are \emph{almost disjoint} if they do not share a common edge.

\end{definition}

\begin{definition}
\label{D:tiling}
A \textit{tiling} of the full $p$-ary infinite tree $T_{p}$ is a pair $(S,\Omega)$ consisting of:
\begin{itemize}
	
	\item A ``tile'' $S$, which is a nontrivial ($E(S) \neq \varnothing$) finite subtree of $T_{p}$.
	
	\item A subset $\Omega \subseteq \Sigma_{p}^{*}$ such that $T_{p}$ is the almost disjoint union of the translates $\omega S$ for $\omega \in \Omega$. We will denote this by
	\[
	T_{p} = \bigsqcup_{\omega \in \Omega} \omega S.
	\]
\end{itemize}
Note that this means that both every vertex of $T_{p}$, i.e. every word $w$, is of the form $w = \omega s$ where $\omega \in \Omega$ and $s \in V(S)$, and also every edge of $T_{p}$ is of the form $\omega e$ where $\omega \in \Omega$ and $e \in E(S)$.
\end{definition}

\begin{figure}[h]
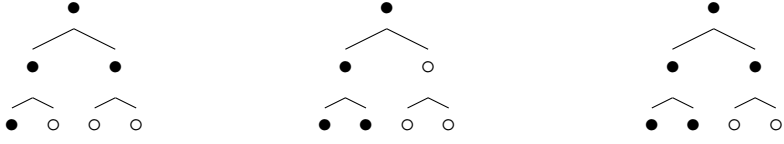

	\begin{minipage}{0.32\textwidth}
		\begin{center}
			\scalebox{1}{\Tree[.$\bullet$ [.$\bullet$ [.$\bullet$ ][.$\circ$ ]]
				[.$\bullet$ [.$\circ$ ][.$\circ$ ]]]}
		\end{center}	
	\end{minipage}%
	\begin{minipage}{0.33\textwidth}
		\begin{center}
			\scalebox{1}{\Tree[.$\bullet$ [.$\bullet$ [.$\bullet$ ][.$\bullet$ ]]
				[.$\circ$ [.$\circ$ ][.$\circ$ ]]]}
		\end{center}	
	\end{minipage}
	\begin{minipage}{0.33\textwidth}
		\begin{center}
			\scalebox{1}{\Tree[.$\bullet$ [.$\bullet$ [.$\bullet$ ][.$\bullet$ ]]
				[.$\bullet$ [.$\circ$ ][.$\circ$ ]]]}
		\end{center}	
	\end{minipage}
	\caption{The first two trees are not tiles, while the third one is.}
\end{figure}

\begin{lemma}
\label{L:tile closed prefixes}
Given a tile $(S,\Omega)$, $S$ is a subtree rooted at $\epsilon$ and is therefore closed under prefixes, namely if $v \in V(S)$ and $w \leq v$ then $w \in V(S)$.
\end{lemma}

\begin{proof}
In a tiling, since the empty word $\epsilon$ must also have the form $\omega s$, we must have $\epsilon \in \Omega \cap V(S)$. In particular a tile $S$ is necessarily a subtree rooted at $\epsilon$, and therefore for any $v \in V(S)$, the unique path from $v$ to the root node is contained in $S$. In terms of words, this means $S$ is closed under prefixes. 
\end{proof}

\begin{lemma}
\label{L:leaf-comparabilty}
For a nontrivial subtree $S$ of $T_{p}$ which is rooted at $\epsilon$, the following properties are equivalent:
\begin{enumerate}[label=(1\alph*)]
	
\item 
\label{item:leaf prefixes or prefixed}
For any word $w \in \Sigma_{p}^{*}$, there is some $\ell \in L(S)$ such that either $w \leq \ell$ or $\ell \leq w$.

\item 
\label{item:nonleaf vertex or leaf prefixes}
For any word $w \in \Sigma_{p}^{*}$, either $w \in V(S) \setminus L(S)$ or $\ell \leq w$ for some $\ell \in L(S)$.

\end{enumerate}
A tree $S$ satisfying these equivalent properties will be called \textit{leaf-comparable}\footnote{The terminology is explained by the fact that property~\ref{item:leaf prefixes or prefixed} says that every word is comparable to an element of $L(S)$ with respect to the partial order $\leq$.}.
\end{lemma}

\begin{proof}
If $\ell \leq w$ for some $\ell \in L(S)$, there is nothing to show; otherwise, observe that by Lemma~\ref{L:tile closed prefixes}, if $w < \ell$ for some $\ell \in L(S)$, then $w \in V(S) \setminus L(S)$ and conversely, if $w \in V(S) \setminus L(S)$ then, since $S$ is a finite subtree, any path in $S$ starting at $w$ and formed by succesive descendants eventually ends in a leaf $\ell$, in which case $w < \ell$.
\end{proof}

\begin{theorem}
\label{T:tile equivalence}
For a nontrivial finite subtree $S$ of $T_{p}$ which is rooted at $\epsilon$, the following properties are equivalent:
\begin{enumerate}
	
\item 
\label{item:tile}
$S$ is a tile in some tiling $(S,\Omega)$.

\item
\label{item:fullness}
$S$ is full, i.e., all $p$ child nodes of any non-leaf node of $S$ belong to $S$.

\item
\label{item:leaf-comparable}
$S$ is leaf-comparable, i.e. for any word $w$ there is a leaf $\ell$ of $S$ such that either $\ell \leq w$ or $w \leq \ell$.

\item
\label{item:unique factorization}
Any $p$-ary word $w$ has a unique factorization $w = \ell_{1} \cdots \ell_{n} s$, where $\braces{\ell_{1}, \ell_{2}, \ldots, \ell_{n}}$ is a (possibly empty) set  of leaves of $S$ and $s$ is a non-leaf node of $S$.
	
\end{enumerate}
Moreover, if these properties hold, then $\Omega$ may be chosen to be the submonoid $\Lambda$ of $\Sigma_{p}^{*}$ generated by the set $L(S)$ of leaves of $S$, that is, $(S,\Lambda)$ is a tiling.
\end{theorem}

\begin{proof}
\leavevmode	
\begin{itemize}
		
\item \eqref{item:tile}~$\implies$~\eqref{item:fullness}:
Suppose that $S$ is a tile which is not full. Then some non-leaf node $s \in  V(S) \setminus L(S)$ does not split completely in $S$, i.e. $sc \notin V(S)$ for some character $c \in \Sigma_{p}$. Choose such an $s$ having minimal length. Thus, if $\sigma \in V(S) \setminus L(S)$ with $\abs{\sigma} < \abs{s}$ then $\sigma$ must split completely in $S$.
		
Consider the edge $(s,sc) \in E(T_{p})$. Since $S$ is a tile, there is some $\omega \in \Omega$ and an edge $(\sigma,\sigma c') \in E(S)$ such that $(s,sc) = \omega(\sigma,\sigma c')$. Then $c = c'$ and $s = \omega \sigma$ with $(\sigma,\sigma c ) \in E(S)$. Since $sc \notin V(S)$, we have $(s,sc) \notin E(S)$ and in particular $\omega \neq \epsilon$. Hence $\abs{\sigma} < \abs{s}$. In addition, since $\sigma c \in V(S)$, $\sigma$ is not a leaf, so $\sigma \in V(S) \setminus L(S)$. Thus $\sigma \Sigma_{p} \subseteq V(S)$. On the other hand, since also $s \in V(S) \setminus L(S)$, $s$ has \emph{some} child, i.e. $sb \in V(S)$ for some $b \in \Sigma_{p}$, so $(s,sb) \in E(S)$. But then $(s,sb) = \omega(\sigma,\sigma b) \in E(S) \cap \omega E(S)$. However, since $\omega \neq \epsilon$, the subtrees $S$ and $\omega S$ are almost disjoint. Thus we arrive at a contradiction which implies that $S$ is indeed full.		
		
\item \eqref{item:fullness}~$\implies$~\eqref{item:leaf-comparable}:
Assume that $S$ is full. Let $w= c_0c_1\cdots c_n \in \Sigma_{p}^{*}$ with each $c_{i} \in \Sigma_{p}$. Recall that $\epsilon \to c_0\to c_0c_1 \to \ldots \to c_0c_1\cdots c_n = w$ is the unique path form $\epsilon$ to $w$. Then $S$ being nontrivial and full ensure that $c_0\in V(S)$. We conclude if $c_0$ is a leaf; otherwise fullness implies  $c_0c_1 \in V(S)$. Arguing recursively, we are left with the case that all of the $c_0c_1\cdots c_i$, with $i < n$, lie in $V(S) \setminus L(S)$. But then $w \in V(S)$,
and there is a path $w c_{n+1} c_{n+2} \cdots$ in $S$ which, since $S$ is finite, eventually exits $S$ at a node with no children, i.e. at a leaf, of which $w$ is a prefix. Thus $S$ is leaf-comparable.

\item \eqref{item:leaf-comparable}~$\implies$~\eqref{item:unique factorization}:
For the existence of the factorization, we may assume that $w \notin V(S) \setminus L(S)$, since otherwise the result is trivial. Since $S$ is leaf-comparable, by Lemma~\ref{L:leaf-comparabilty}~\ref{item:nonleaf vertex or leaf prefixes}
there is some leaf $\ell_{1}$ of $S$ such that $\ell_{1} \leq w$. Thus $w = \ell_{1} w_{1}$ for some word $w_{1}$. If $w_{1} \in V(S) \setminus L(S)$, we are done, otherwise there is some $\ell_{2} \in L(S)$ such that $\ell_{2} \leq w_{1}$ and then $w = \ell_{1} \ell_{2} w_{2}$ for some $w_{2}$. Since this process cannot continue ad infinitum, we eventually arrive at a factorization $w = \ell_{1} \cdots \ell_{n} s$ with $\ell_{i} \in L(S)$ and $s \in V(S) \setminus L(S)$.

For uniqueness, suppose we have two factorizations
\begin{equation}
\label{E:factorizations}
	\ell_{1} \cdots \ell_{n} s_{1} = \ell_{1}' \cdots \ell_{m}' s_{2},
	\tag{$\ast$}
\end{equation}
where $n,m \geq 0$, each $\ell_{i}, \ell'_{j} \in L(S)$, and each $s_{j} \in V(S) \setminus L(S)$. We imitate the proof of unique factorization of integers into primes.
If $m = n = 0$ the conclusion is obvious. Suppose now that $m,n \geq 1$. Without loss of generality, we may assume that $\abs{\ell_{1}} \leq \abs{\ell_{1}'}$. Comparing the leftmost parts of both sides of~\eqref{E:factorizations}, we conclude that in fact $\ell_{1} \leq \ell_{1}'$ and, since both are leaves of $S$, $\ell_{1} = \ell_{1}'$. Cancelling these terms leaves a shorter relation. We must have $n = m$, otherwise eventually we would arrive at a relation of the form $s_{1} = \ell_{1}'' \cdots \ell_{k}'' s_{2}$ or $s_{2} = \ell_{1}'' \cdots \ell_{k}'' s_{1}$ with $k \geq 1$ and $\ell_{j}'' \in L(S)$, which also correspond to the cases where $m = 0$ or $n = 0$ but not both. In the first case $\ell_{1}'' \leq s_{1}$, which, since $\ell_{1}'' \in L(S)$ and $s_{1} \in V(S)$, can only hold if $\ell_{1}'' = s_{1}$, contradicting $s_{1} \notin L(S)$, and similarly in the second case. Thus $n = m$ in all cases and successive cancellations lead to $\ell_{j} = \ell_{j}'$ for all $j$ and $s_{1} = s_{2}$.

\item \eqref{item:unique factorization}~$\implies$~\eqref{item:leaf-comparable}:
This only depends on the existence of the factorization. Given a word $w$, if its factorization begins with some leaf $\ell$, then $\ell \leq w$. Otherwise $w \in V(S) \setminus L(S)$, in which case following a path downward (away from the root node $\epsilon$) from $w$ we eventually exit $S$ at a leaf node $\ell$, hence $w \leq \ell$.

\item \eqref{item:leaf-comparable}~$\implies$~\eqref{item:fullness}:
Suppose that $S$ is leaf-comparable but not full. Let $s$ be a non-leaf node of $S$ which does not split completely. Suppose $c \in \Sigma_{p}$ with $w = sc \notin V(S)$. If $w \leq \ell$ for some leaf $\ell$ of $S$, then the unique path from $s$ to $\ell$ is not contained in $S$, contradicting connectedness. If $\ell \leq w$, then in fact $\ell \leq s$, but this forces $\ell = s$, contradicting that $s$ is not a leaf. Thus $w$ is a word which is not comparable to any leaf, contradicting~\eqref{item:leaf-comparable}.

\item \eqref{item:unique factorization}~$\implies$~\eqref{item:tile}:
Finally, we show that if unique factorization holds, then in fact $(S,\Lambda)$ is a tiling, where $\Lambda$ is the submonoid of $\Sigma_{p}^{*}$ generated by $L(S)$.

Let $w \in \Sigma_{p}^{*}$ be any word and consider any edge $(w,wc) \in E(T_{p})$, with $c \in \Sigma_{p}$. We have already seen that $S$ must be full. Write $w = \lambda s$ with $\lambda \in \Lambda$ and $s \in V(S) \setminus L(S)$. Fullness implies $sc \in V(S)$, hence $w \in \lambda V(S)$ and $(w,wc) = \lambda(s,sc) \in \lambda E(S)$.
		
It remains to show that the translates $\lambda S$ are almost disjoint. If $\alpha,\beta \in \Lambda$ with $\alpha E(S) \cap \beta E(S) \neq \varnothing$, there are edges $(s_{j},s_{j}c_{j}) \in E(S)$ for $j=1,2$, such that $\alpha (s_{1},s_{1} c_{1}) = \beta (s_{2}, s_{2} c_{2})$. Then $s_{j} \in V(S) \setminus L(S)$ and $\alpha s_{1} = \beta s_{2}$. By uniqueness of the factorization, we must have $\alpha = \beta$. 
\qedhere
\end{itemize}
\end{proof}

We end this section by showing that, although in the definition of a tiling $(S,\Omega)$ (Definition~\ref{D:tiling}), $\Omega$ is only assumed to be a subset of the monoid of words $\Sigma_{p}^{*}$, in fact $\Omega$ is uniquely determined by $S$ as its ``monoid of leaves'', namely, the submonoid $\Lambda$ of $\Sigma_{p}^{*}$ generated by the set $L(S)$ of leaves of $S$.

\begin{proposition}
\label{P:monoid is unique}
If $(S,\Omega)$ is a tiling of $T_{p}$, then in fact $\Omega$ must be the submonoid $\Lambda$ of $\Sigma_{p}^{*}$ generated by the set $L(S)$ of leaves of $S$.
\end{proposition}

\begin{proof}
Let $\Lambda$ be the submonoid of $\Sigma_{p}^{*}$ generated by the leaves of $S$. By Theorem~\ref{T:tile equivalence}, we know that $T_{p} = \bigsqcup_{\lambda \in \Lambda} \lambda S$. We have to show that $\Lambda$ is the only \emph{subset} of $\Sigma_{p}^{*}$ that satisfies this.
\begin{itemize}

\item It is enough to show that $\Omega \subseteq \Lambda$. Assume that this is the case. Then given $\lambda \in \Lambda$, choose any character $c \in \Sigma_{p}$. There is some $\omega \in \Omega$ such that $(\lambda,\lambda c) \in \omega E(S)$. Write $(\lambda,\lambda c) = \omega (s,sb)$ with $(s,sb) \in E(S)$. Thus $s \in V(S) \setminus L(S)$. Then $\lambda = \omega s$ and by unique factorization (Theorem~\ref{T:tile equivalence}~\eqref{item:unique factorization}), we must have $s = \epsilon$, so that $\lambda = \omega \in \Omega$. Hence $\Lambda \subseteq \Omega$.

\item $\Omega \subseteq \Lambda$: assume the contrary and choose $\omega \in \Omega \setminus \Lambda$ with minimal length. We have $\omega = \lambda s$ for some $\lambda \in \Lambda$ and $s \in V(S)$. By assumption, $s \notin L(S)$ and $s \neq \epsilon$.

Let us show that $\lambda \in \Omega$. Consider any edge $(\lambda,\lambda c) \in E(T_{p})$. Let $\omega' \in \Omega$ such that $(\lambda, \lambda c) = \omega'(\sigma, \sigma c')$ with $(\sigma,\sigma c') \in E(S)$. Then $c = c'$ and $\lambda = \omega' \sigma$. Now $\abs{\omega'} \leq \abs{\lambda} < \abs{\omega}$, so $\omega' \in \Lambda$ by minimality. Since $\lambda, \omega' \in \Lambda$ and $\sigma \in V(S) \setminus L(S)$, the relation $\lambda = \omega' \sigma$ implies by unique factorization that $\sigma = \epsilon$ and $\lambda = \omega' \in \Omega$.

Now we have $\omega = \lambda s$ with $\lambda \in \Lambda \cap \Omega$ and $s \in V(S) \setminus L(S)$. Thus there is some $c \in \Sigma_{p}$ with $(s,sc) \in E(S)$. But then, since $S$ is a tile, and hence full by Theorem~\ref{T:tile equivalence}, $(\omega,\omega c) = \lambda (s, sc) = \omega (\epsilon,c) \in \lambda E(S) \cap \omega E(S)$. Since $\lambda, \omega \in \Omega$, almost disjointness implies that $\omega = \lambda \in \Lambda$, contradicting the choice of $\omega$. Thus $\Omega \subseteq \Lambda$.
\qedhere
\end{itemize}
\end{proof}


\section{Homeomorphisms of $p$-adic spaces}
\label{sec:homeomorphisms}

A \emph{Cantor space} is a compact, metrizable, totally disconnected, perfect topological space. The classification theorem for Cantor spaces is due to Brouwer, who proved that all Cantor spaces are homeomorphic to the classical Cantor set~\cite{Brouwer}. In particular, any two Cantor spaces are homeomorphic.

Brouwer's Theorem can be proved in a fairly constructive way by associating an infinite tree to a given Cantor space. The construction may be understood in two steps, the first of which is to exhibit a homeomorphism of a given Cantor space with a tree, and then using this to relate two given Cantor spaces to each other.

Assigning a tree to a Cantor space relies on the existence of finite clopen partitions of its nonempty clopen subsets.\footnote{By partition we mean a subdivision into pairwise disjoint \emph{nonempty} subsets, and ``clopen partition'' refers to the sets in the partition being themselves clopen.} Such a set has a clopen partition with any given number $n$ of elements. Choosing a large enough $n$, the diameters (with respect to a metric inducing the topology) of the sets in the partition may be taken to be  arbitrarily small.

With this in mind, for a given Cantor space $X$, one can recursively construct a
sequence of finite clopen partitions $\mathcal{U}_{n}$ such that each $\mathcal{U}_{n}$
refines $\mathcal{U}_{n-1}$ and the maximum diameter of the sets in $\mathcal{U}_{n}$
tends to $0$ as $n \to \infty$. A point $x \in X$ then has a unique associated
``container sequence'' $(U_{0},U_{1},\ldots)$ such that $x \in U_{n} \in
\mathcal{U}_{n}$, with $\braces{x} = \bigcap_{n=0}^{\infty} U_{n}$. Considering each
$\mathcal{U}_{n}$ as a finite topological space with the discrete topology, the inverse
limit space $\mathcal{U}_{\infty} = \varprojlim \mathcal{U}_{n}$ with respect to the
container maps $c_{n} : \mathcal{U}_{n} \to \mathcal{U}_{n-1}$ sending a set $U \in
\mathcal{U}_{n}$ to its unique container $V \in \mathcal{U}_{n-1}$, is homeomorphic to
$X$. Clearly $\mathcal{U}_{\infty}$ is an infinite tree, with vertices the sets in the
partitions and edges given by the containment relations among these.

 
Given two Cantor spaces $X$, $Y$, we can recursively construct two such sequences of partitions, one for each space respectively, not only having the same number of elements at each step, but also isomorphic inclusion relations, i.e., the partitions have associated graphs which are isomorphic, and hence define a homeomorphism of $X$ and $Y$ by choosing compatible bijections between the partitions at each level.\footnote{The details of what we are summarizing here may be found in Willard's well-known textbook~\cite{Willard}, for example.}

The $p$-adic integers $\mathbb{Z}_{p}$, for a given prime $p$, provide concrete examples of Cantor spaces having additional algebraic structure. From our point of view, they are convenient because the $p$-adic metric on the integers is easy to describe, and the correspondence between Cantor spaces and trees can be easily seen to pair $\mathbb{Z}_{p}$ with the full infinite $p$-ary tree $T_{p}$. The results of Section~\ref{sec:tilings}  yield a constructive method for obtaining homeomorphisms between $\mathbb{Z}_{p}$ and $\mathbb{Z}_{q}$ for any pair of primes $p,q$. Thus, from now on we shall assume that $p$ is prime for simplicity.\footnote{Note that here we are disregarding the algebraic structure. There are no non-trivial additive group homomorphisms $\mathbb{Z}_{p} \to \mathbb{Z}_{q}$. Much of what we say is valid for $p$-ary trees for any \emph{integer} $p$, but the $p$-adic metric for prime $p$ is easier to describe. See Remark~\ref{R:pq nonprime}.}

\begin{remark}
\label{R:pq nonprime}
The reader may have noticed that the primality of $p$ is not necessary. Indeed everything holds for $a$-ary trees where $a$ is a natural number $\geq 2$. The Chinese Remainder Theorem and general properties of inverse limits show that $\mathbb{Z}_{a} = \varprojlim \mathbb{Z}/a^{n}\mathbb{Z}$ is isomorphic to $\prod_{p \mid a} \mathbb{Z}_{p}$ over prime divisors of $a$. In particular, $\mathbb{Z}_{p^{\ell}}$ is isomorphic to $\mathbb{Z}_{p}$ for primes $p$. The latter isomorphism corresponds to considering a $p^{\ell}$-adic digit as a block of $\ell$ $p$-adic digits. On trees this means mapping a given $p^{\ell}$-adic subtree to the $p$-adic subtree determined by expanding into blocks and filling in with all ancestors. In this sense the prime case is enough to understand the rest.
\end{remark}

A well-known method of proving Brouwer's Theorem and other universality results involving the Cantor set is to use inverse limit spaces and may be found in many sources, for example~\cite{Willard}. Following the proof, which we have sketched above, we will construct such an explicit homeomorphism by means of tilings of the trees $T_{p}$ and $T_{q}$. Let us start with the topological interpretation of the results from section~\ref{sec:tilings}.

To a word $w = a_{0}a_{1} \cdots a_{n-1} \in \Sigma_{p}^{*}$ we associate the natural number $N(w) = a_{0} + a_{1} p + \cdots + a_{n-1} p^{n-1} \in [0,p^{n})$ and the $p$-adic ball $\mathbb{B}_{w} = B(N(w), p^{-\abs{w}+1})$ in $\mathbb{Z}_{p}$. The empty word $\epsilon$ is assigned the ball $\mathbb{B}_{\epsilon} = \mathbb{Z}_{p}$. It can be easily shown that every $p$-adic ball is of the form $\mathbb{B}_{w}$ for some word $w \in \Sigma_{p}^{*}$ (see~\cite{Robert}).

The relevance of word operations to the problem of constructing covers by balls is a consequence of the following simple observation.

\begin{lemma}
\label{L:to contain is to prefix}
For two words $v,w \in \Sigma_{p}^{*}$, $v \leq w$ if and only if $\mathbb{B}_{v} \supseteq \mathbb{B}_{w}$. In particular, $\mathbb{B}_{v} \cap \mathbb{B}_{w} \neq \varnothing$ if and only if $v$ and $w$ are comparable, this is, either $v \leq w$ or $w \leq v$.
\end{lemma}

\begin{proof}
Let $v = a_0a_1\cdots a_r$ and $w = b_0b_1\cdots b_s$ in $\Sigma_{p}^{*}$. Since every point of a $p$-adic ball is a center, the inclusion $B(N(w),p^{-m+1}) \subseteq B(N(v), p^{-n+1})$ holds iff $\abs{N(w) - N(v)}_{p} < p^{-n+1}$ and $B(N(w),p^{-m+1}) \subseteq B(N(w), p^{-n+1})$ iff $n \leq m$ and $N(w) \equiv N(v) \bmod p^{n}$ iff $n \leq m$ and  $a_{i} = b_{i}$ for $0 \leq i \leq n-1$, i.e. $v \leq w$. The particular case follows from the fact that two non-disjoint $p$-adic balls satisfy an inclusion relation.
\end{proof}

\begin{lemma}
\label{L:leaves cover equivalent to words version}
A nontrivial subtree $S$ of $T_{p}$ which is rooted at $\epsilon$ is leaf-comparable if and only if the balls $\braces{\mathbb{B}_{\ell} : \ell \in L(S)}$ corresponding to the leaves of $S$ form a partition of $\mathbb{Z}_{p}$.
\end{lemma}

\begin{proof}
	Assume that $\braces{\mathbb{B}_{\ell} : \ell \in L(S)}$ is a partition of $\mathbb{Z}_{p}$ and let us show that the property~\ref{item:leaf prefixes or prefixed} from Lemma~\ref{L:leaf-comparabilty} holds. Let $w \in \Sigma_{p}^{*}$. Since $\mathbb{B}_{w} \subseteq \mathbb{Z}_{p} = \bigsqcup_{\ell \in L(S)} \mathbb{B}_{\ell}$, we must have $\mathbb{B}_{w} \cap \mathbb{B}_{\ell} \neq \varnothing$ for some $\ell \in L(S)$, hence either $w \leq \ell$ or $\ell \leq w$.
	
Conversely, the union $U = \bigsqcup_{\ell \in L(S)} \mathbb{B}_{\ell}$ is a clopen subset of $\mathbb{Z}_{p}$. If $U \subsetneq \mathbb{Z}_{p}$, then $\mathbb{B}_{w} \cap U = \varnothing$ for some $w \in \Sigma_{p}^{*}$. Since $\mathbb{B}_{w} \cap \mathbb{B}_{\ell} = \varnothing$ for every $\ell \in L(S)$, neither $w \leq \ell$ nor $\ell \leq w$ for every leaf $\ell \in L(S)$.
\end{proof}

\begin{theorem}
\label{T:tiling topological characterization}
For a nontrivial subtree $S$ of $T_{p}$ which is rooted at $\epsilon$, the following properties are equivalent:
\begin{enumerate}[label=(\arabic*)]

\item 
\label{item:S is tile}
$S$ is a tile in some tiling of $T_{p}$.

\item
\label{item:leaves are partition}
The balls $\braces{\mathbb{B}_{\ell} : \ell \in L(S)}$ form a partition of $\mathbb{Z}_{p}$.

\end{enumerate}
\end{theorem}

\begin{proof}
This is an immediate corollary of Lemma~\ref{L:leaves cover equivalent to words version} and Theorem~\ref{T:tile equivalence}.
\end{proof}

We end this section by showing how these equivalences allow us to build an inverse system of partitions from a given tile and, in turn, to construct a homeomorphism between $\mathbb{Z}_{p}$ and $\mathbb{Z}_{q}$ based on two systems of partitions with isomorphic inclusion relations.

Consider now a tile $S$ of $T_{p}$. By Theorem~\ref{T:tiling topological characterization}, the set of balls corresponding to leaves, $\braces{\mathbb{B}_{\ell} : \ell \in L(S)}$, form a partition of $\mathbb{Z}_p$ with cardinality $\# L(S) = 1+(p-1)s$ for some $s\in \mathbb{Z}_{>0}$. This follows because $S$ must be full, and a full subtree can be built recursively by choosing at each step which old leaves split completely to form $p$ new leaves at the next level. Each lost old leaf gives rise to $p$ new leaves, thus adding $p-1$ leaves to the overall count. Here $s$ is the number of nodes which have split during this process and also the final number of non-leaf nodes in $S$: $s = \#(V(S)\setminus L(S))$.

The translates of $S$ by the monoid $\Lambda$ generated by $L(S)$ form a tilling of $\mathbb{Z}_{p}$, and it is easy to see that for each $n\geq 1$ the tree $S_{n} \coloneqq \braces{\ell_1\ell_2\cdots \ell_iw \colon 1\leq i < n, \ell_j \in L(S), w \in S}$ is again a tile of $\mathbb{Z}_{p}$ with $L(S_{n}) = \braces{\ell_1\cdots \ell_{n} \colon \ell_1, \ldots, \ell_{n}\in L(S)}$ by unique factorization (Theorem~\ref{T:tile equivalence}), in particular $\# L(S_n) = (\# L(S))^{n}$. In terms of balls, this means that the partition $\braces{\mathbb{B}_{\lambda} : \lambda \in L(S_{n+1})}$ of $\mathbb{Z}_{p}$ is obtained from $\braces{\mathbb{B}_{\lambda} : \lambda \in L(S_n)}$ by subdividing each ball $\mathbb{B}_{\lambda}$ with $\lambda \in L(S_n)$ into $\# L(S)$ smaller balls, namely into $\braces{\mathbb{B}_{\lambda\ell} \colon \ell \in L(S)}$.

Suppose that we are given primes $p$ and $q$ and tiles $S$ of $T_{p}$ and $S'$ of $T_{q}$ such that $\# L(S) = \# L(S')$. This is always possible since the Diophantine equation $1+(p-1)s = 1+(q-1)s'$ always admits a solution for integers $p,q$. Let us choose a bijection $f\colon L(S)\to L(S')$. Then $f$ clearly extends to a bijection $f_n\colon L(S_n) \to L(S'_n)$ for all $n\geq 1$ by setting $f_n(\ell_1\ell_2\cdots\ell_n) \coloneqq f(\ell_1)f(\ell_2)\cdots f(\ell_n)$, which in turn defines bijections $g_n$ between the partitions $\braces{\mathbb{B}_{\ell} : \ell \in L(S_n)}$ of $\mathbb{Z}_p$ and $\braces{\mathbb{B}_{\ell'} : \ell' \in L(S'_n)}$ of $\mathbb{Z}_q$ for all $n\geq 1$ by setting $g_n(\mathbb{B}_{\ell}) \coloneqq \mathbb{B}_{f_n(\ell)}$.

Notice that the bijections $g_n$ preserves inclusions, namely given $\ell \in L(S_n)$ and $k\in L(S_{n+1})$, then by Lemma~\ref{L:leaves cover equivalent to words version} we have that
\[
\mathbb{B}_k \subseteq \mathbb{B}_{\ell} \iff k\geq \ell \iff f_{n+1}(k)\geq f_n(\ell) \iff g_{n+1}(\mathbb{B}_k) \subseteq g_n(\mathbb{B}_{\ell}).
\]
This property ensures that the bijections $g_n$ define a homeomorphism between $\mathbb{Z}_p$ and $\mathbb{Z}_q$.

\section{Examples}
\label{sec: examples}

We conclude with several examples of tilings inducing homeomorphisms between $p$-adic spaces, applying the results of the previous section and giving the corresponding action on digits.

We begin by exhibiting a concrete homeomorphism between $\mathbb{Z}_2$ and $\mathbb{Z}_3$ arising from the tiles $S$ and $S'$ shown in Figures~\ref{fig: tile S of Z_2} and~\ref{fig: tile S' of Z_3}. The sets of leaves are $L(S) = \braces{\lambda_1 = 00, \lambda_2 = 01, \lambda_3 = 1}\subset \Sigma_{2}^{*}$ and $L(S') = \braces{\mu_1 = 0, \mu_2 = 1, \mu_3 = 2} \subset \Sigma_{3}^{*}$. Figures~\ref{fig: S tilling} and \ref{fig: S' tilling} shows how the translates of $S$ and $S'$ by the monoids $\Lambda$ and $\Lambda'$ generated by the respective sets of leaves $L(S)$ and $L(S')$, cover the trees $T_2$ and $T_3$ respectively.

\begin{figure}[h]
	\centering
	\begin{minipage}{.5\textwidth}
		\centering
		\centering
		\scalebox{0.8}{\Tree[.$\epsilon$ 	[.$v$ [.$\lambda_1$ ][.$\lambda_2$ ]]
			[.$\lambda_3$ ]]}
		\caption{A tile $S$ of $T_{2}$.}
		\label{fig: tile S of Z_2}
	\end{minipage}%
	\begin{minipage}{.5\textwidth}
		\centering
		\scalebox{1}{\Tree[.$\epsilon$ 		[.$\mu_1$ ]
			[.$\mu_2$ ]
			[.$\mu_3$ ]]}
		\caption{A tile $S'$ of $T_{3}$.}
		\label{fig: tile S' of Z_3}
	\end{minipage}
\end{figure}

\begin{center}
	\begin{figure}[h]
		\scalebox{0.8}{\Tree[.$\epsilon$ 	[.$v$ [.$\lambda_1$ [.$\lambda_1v$ [.$\lambda_1\lambda_1$ ][.$\lambda_1\lambda_2$ ]][.$\lambda_1\lambda_3$ ]][.$\lambda_2$ [.$\lambda_2v$ [.$\lambda_2\lambda_1$ ][.$\lambda_2\lambda_2$ ]][.$\lambda_2\lambda_3$ ]]]
			[.$\lambda_3$ [.$\lambda_3v$ [.$\lambda_3\lambda_1$ ][.$\lambda_3\lambda_2$ ]][.$\lambda_3\lambda_3$ ]]]}
		\caption{Part of the tilling of $T_{2}$ by $S$}
		\label{fig: S tilling}
	\end{figure}
\end{center}\vspace{0.5cm}

\begin{figure}[h]
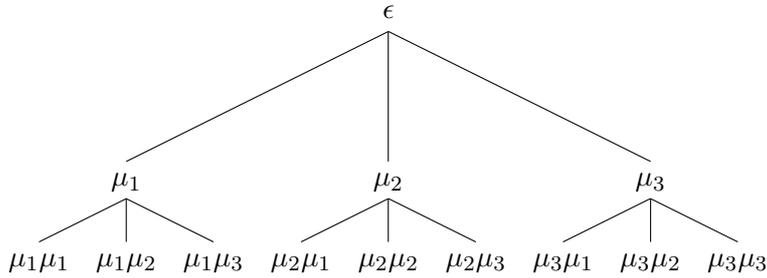

	\centering
	\scalebox{1}{\Tree[.$\epsilon$ 		[.$\mu_1$ [.$\mu_1\mu_1$ ]
		[.$\mu_1\mu_2$ ]
		[.$\mu_1\mu_3$ ]]
		[.$\mu_2$ [.$\mu_2\mu_1$ ]
		[.$\mu_2\mu_2$ ]
		[.$\mu_2\mu_3$ ]]
		[.$\mu_3$ [.$\mu_3\mu_1$ ]
		[.$\mu_3\mu_2$ ]
		[.$\mu_3\mu_3$ ]]]}
	\caption{Part of the tilling of $T_{3}$ by $S'$}
	\label{fig: S' tilling}
\end{figure}

Consider now the bijection $f\colon L(S) \to L(S')$ with $f(\lambda_i) = \mu_i$ for $i=1,2,3$, which defines a homeomorphism $\phi\colon \mathbb{Z}_2 \to \mathbb{Z}_3$. In concrete terms, given $x \in \mathbb{Z}_2$ we consider its $2$-adic expansion $w$, which is an infinite string in the alphabet $\Sigma_{2}$, then the $3$-adic expansion of $\phi(x)$ is obtained from $w$ by applying $f$ to its prefixes $\lambda_i$. In other words, one must identify inside $w$ the blocks of digits corresponding to each of the words $\lambda_i$, and transform these blocks according to $f$. For example, the image of $w=0001101\ldots=\lambda_1\lambda_2\lambda_3\lambda_2\ldots$ is $\mu_1\mu_2\mu_3\mu_2\ldots = 0121\ldots$

Finally, we give example of several auto-homeomorphisms of $\mathbb{Z}_2$ arising from choosing different tilings of $T_2$. Recall that, since $p=2$, $\#L(S) = 1+s$ with $s=\#(V(S)\setminus L(S)) \in \mathbb{Z}_{>0}$. If we take $s=2$, there are only two such tiles, shown in Figures~\ref{fig: tile T2 s=2} and~\ref{fig: tile T2 s=2 sim}, which moreover are isomorphic as trees. Taking for example the bijection $f(\lambda_i) = \mu_i$, we obtain the homeomorphism of $\mathbb{Z}_2$ mapping the blocks: $00$ to $11$, $01$ to $10$ and $1$ to $0$; which is nothing but the transposition of the digits $0$ and $1$. If we take the bijection $f(\lambda_1) = \mu_2$, $f(\lambda_2) = \mu_1$ and $f(\lambda_3) = \mu_3$, the corresponding homemorphism is no longer a permutation since on blocks it maps $00$ to $10$, $01$ to $11$ and $1$ to $0$.

\begin{figure}[h]
\centering
\begin{minipage}{.5\textwidth}
\centering
\scalebox{1}{\Tree[.$\epsilon$ [.$v$ [.$\lambda_1$ ][.$\lambda_2$ ]]
[.$\lambda_3$ ]]}
\caption{A tile $S$ of $T_2$}
\label{fig: tile T2 s=2}
\end{minipage}%
\begin{minipage}{.5\textwidth}
\centering
\scalebox{1}{\Tree[.$\epsilon$ [.$\mu_3$ ][.$w$ [.$\mu_2$ ][.$\mu_1$ ]]
]}
\caption{A tile $S'$ of $T_2$}
\label{fig: tile T2 s=2 sim}
\end{minipage}
\end{figure}

For $s=3$, we can consider the tiles from Figures~\ref{fig: tile T2 s=3 A} and~\ref{fig: tile T2 s=3 B}, which are no longer isomorphic as trees. Choosing $f(\lambda_i) = \mu_i$, we get the map on blocks of digits given by: $00$ to $000$, $01$ to $001$, $10$ to $01$ and $11$ to $1$. For example, $\lambda_1\lambda_2\lambda_3\lambda_4\lambda_2\ldots = 0001101101\ldots$ is mapped to $\mu_1\mu_2\mu_3\mu_4\mu_2\ldots = 000001011001\ldots$

\begin{figure}[h]
\centering
\begin{minipage}{.45\textwidth}
\centering
\scalebox{1}{\Tree[.$\epsilon$ [.$v$ [.$\lambda_1$ ][.$\lambda_2$ ]][.$w$ [.$\lambda_3$ ][.$\lambda_4$ ]]
]}
\caption{A tile of $T_2$}
\label{fig: tile T2 s=3 A}
\end{minipage}%
\begin{minipage}{.55\textwidth}
\centering
\scalebox{1}{\Tree[.$\epsilon$ [.$v'$ [.$w'$ [.$\mu_1$ ][.$\mu_2$ ]][.$\mu_3$ ]]
[.$\mu_4$ ]]}
\caption{A tile of $T_2$}
\label{fig: tile T2 s=3 B}	
\end{minipage}
\end{figure}




\end{document}